\documentclass{amsart}
\usepackage{a4wide,amssymb}

\def\1{^{-1}}

\newtheorem{theorem}{Theorem}[section]

\newtheorem{proposition}[theorem]{Proposition}

\theoremstyle{definition}

\theoremstyle{remark}

\numberwithin{equation}{section}

\begin{document}
\title[On mod $p$ Honda formal group law]{Polynomial behavior of the Honda formal group law}

\begin{abstract}
This note provides the calculation of the formal group law $F(x,y)$ in modulo $p$ Morava $K$-theory at prime $p$ and $s>1$ as an element in $K(s)^*[x][[y]]$ and some applications to relevant examples.
\end{abstract}
\author{Malkhaz Bakuradze }
\address{Faculty of exact and natural sciences, A.Razmadze Math. Institute of Iv. Javakhishvili Tbilisi State University, Georgia }
\email{malkhaz.bakuradze@tsu.ge }
\thanks{The author was supported by Volkswagen Foundation, Ref.: I/84 328.}
\commby{Communicated by V. V. Vershinin}
\maketitle
\section{Introduction}

Let $K(s)^*(-)$, be the $s$-th Morava $K$-theory at prime $p$ \cite{JW}. The coefficient ring $K(s)^*(pt)$ is the Laurent polynomial ring in one variable $\mathbb{F}_p[v_s,v_s^{-1}]$, where $\mathbb{F}_p$ is the field of $p$ elements and $deg(v_s)=-2(p^s-1)$.

\medskip

Let $F(x,y)$ be the formal group law in $K(s)^*(-)$ theory \cite{H}. The purpose of this note is to prove that if $s>1$, the formal group law $F(x,y)$ is a polynomial modulo $y^{p^{i(s-1)}}$ (or equivalently modulo $x^{p^{i(s-1)}}$) for any $i\geq 1$, see Theorem \ref{Honda}. This fact was never mentioned before in the literature even though the proof is quite simple.  We also want to have a method for explicit calculation. The idea is to apply the Ravenel formula \cite{R} involving Witt's symmetric polynomials. The proof does not work for $s=1$.
%Moreover it seems the statement is false in the latter case.
The particular case \eqref{n=1} of Theorem \ref{Honda} was applied in several papers by the author, also \cite{SCH1, SCH2, SCH3}, and \cite{KLW}. Some other motivation is given in Section 3.
\bigskip

\section{The statement}

Recall the recursive formula from Ravenel's green book, (see \cite{R}, 4.3.8) for the formal group law. In $K(s)^*(-)$ theory it reads  (we set $v_s=1$ and  $q=p^{s-1}$ )

\begin{equation}
\label{eq:Ravenel}
F(x,y)=F(x+y,w_1(x,y)^{q},w_2(x,y)^{q^{2}},w_3(x,y)^{q^{3}},\cdots )
\end{equation}
where $F(x,y,z,\cdots)=x\oplus_{F}y\oplus_{F}z\oplus_{F}\cdots $ is the iterated $x\oplus_{F} y=F(x,y)$ and $w_j$ are $\mod(p)$ Witt's integral symmetric homogeneous polynomials of degree $p^j$ :

$$
x^{p^n}+y^{p^n}=\sum_jp^jw_j(x,y)^{p^{n-j}}.
$$

In particular

\begin{align*}
w_0=&x+y,\\
w_1=&-\sum_{0<j<p}p^{-1}\binom{p}{j}x^{j}y^{p-j}.\\
\end{align*}

We will need that $deg(w_j)=p^j$ and that $w_j(x,0)=w_j(0,y)=0$, for $j>0$.

\medskip

Clearly we have

\begin{equation}
\label{n=0}
F(x,y)=x+y \text{ modulo } y^{q}.
\end{equation}

One has for $s>1$ (see \cite{BP})

\begin{equation}
\label{n=1}
F(x,y)=x+y+w_1(x,y)^{q}\text{ modulo } y^{q^{2}}.
\end{equation}

We now want to prove that for $s>1$, $F(x,y)$ is again a polynomial modulo $y^{q^{n}}$ for any $n$.
The idea is to apply \eqref{eq:Ravenel}.

\bigskip

\begin{theorem}
\label{Honda}
 One has
$F(x,y)\in K(s)^*[x][[y]]$
for the formal group law $F(x,y)$ in mod $p$ Morava $K(s)^*(-)$ theory at $p$ and $s>1$.

A method for calculation of $F(x,y)$ modulo $y^{q^{n}}$ is given by the Ravenel formula \eqref{eq:Ravenel} and induction on $n$.
\end{theorem}

\begin{proof}  By induction hypothesis, we have that $F(x,y)$ modulo $y^{q^k}$, $k\leq n$ is a polynomial, say $P_k(x,y)$. By \eqref{n=0} and  \eqref{n=1} we have
	
\begin{equation*}
P_1(x,y)=x+y, \,\,\,P_2(x,y)=x+y+w_1(x,y)^q.
\end{equation*}

Induction step: Let us work  modulo $y^{q^{n+1}}$. Then \eqref{eq:Ravenel} implies

\begin{equation*}
\label{1}
F(x,y)\equiv F(x+y, w_1(x,y)^{q},w_2(x,y)^{q^2},\cdots , w_n(x,y)^{q^n}).
\end{equation*}

By induction hypothesis we have

\begin{align*}
	\label{2}
	F(x,y)\equiv & z_1+w_n^{{q^n}}, \text{\,\,\,\,\,\,\,\,\,\,\,\, where \, }  z_1=F(x+y,w_1^{q},\cdots , w_{n-1}^{q^{n-1}});\\
	 z_1\equiv& P_{2}(z_2, w_{n-1}^{q^{n-1}}), \,\,\,\,\,\,\,\,\,\,\,\,\,\,\,\,\,\,\,\,\,\, z_2=F(x+y,w_1^{q},\cdots, w_{n-2}^{q^{n-2}});\\
	z_2\equiv& P_{3}(z_3, w_{n-2}^{q^{n-2}}),
	\,\,\,\,\,\,\,\,\,\,\,\,\,\,\,\,\,\,\,\,\,\,\, z_3=F(x+y,w_1^{q},\cdots, w_{n-3}^{q^{n-3}});\\
	&\cdots  & \\
	z_{n-2}\equiv & P_{n-1}(z_{n-1},w_2^{q^2}),\,\,\,\,\,\,\,\,\,\,\,  z_{n-1}=F(x+y,w_1^q);\\
	z_{n-1}\equiv & P_{n}(x+y,w_1^q).&\\
\end{align*}

 Accordingly $F(x,y)$ is again a polynomial modulo $y^{q^{n+1}}$ for any natural $n$.  Therefore one can collect the coefficients at $y^j$, $j<q^{n+1}$ for any $n$ and write
$$
F(x,y)=\sum A_l(x)y^l\in K(s)^*[x][[y]].
$$
\end{proof}

The proof above gives more, namely one can evaluate the degree of the polynomial $A_l(x)$ .

\bigskip

\begin{proposition}
	\label{prop}
In $F(x,y)=\sum \alpha_{ij}x^iy^j$ we have $\alpha_{ij}=0$ for $i>(pq)^n$ whenever $j< q^{n}.$
\end{proposition}

\begin{proof}
Base case is obvious. Induction step:
By Theorem \ref{Honda} we have modulo $y^{q^{n+1}}$
	
\begin{equation*}
F(x,y)\equiv P_{n}(x+y,z)+w_n(x,y)^{q^{n}},\,\,\,\,\,\,
z=F(w_1(x,y)^q,\cdots,w_{n-1}(x,y)^{q^{n-1}}).
\end{equation*}

\medskip

The term $w_n^{q^n}$ is of degree $(pq)^n$ hence is irrelevant.

 Let $\beta_{ij}(x+y)^{i}z^{j}$ be any term of the polynomial $P_{n}(x+y,z)$. By induction hypothesis we have
 $$i\leq(pq)^n \text{ whenever } j< q^{n}.$$

Then $z^j$ is a polynomial in $w_1^q,\cdots, w_{n-1}^{q^{n-1}}$. Therefore it has the terms
$$
(w_1^q)^{j_1}\cdots (w_{n-1}^{q^{n-1}})^{j_{n-1}},\text{ with \,\,\,} j_1<q^n,\cdots j_{n-1}<q^2,
$$

as we work modulo $y^{q^{n+1}}$.

Therefore $z^j$, as a polynomial in $x$ and $y$, has the terms of total degree
$$pqj_1+(pq)^2j_2+\cdots +(pq)^{n-1}j_{n-1}<pq^{n+1}+p^2q^{n+1}+\cdots +p^{n-1}q^{n+1}.$$

Thus for any term of $P_n(x+y,z)$
$$\text{the total degree}
<(pq)^{n}+q^{n+1}\sum_{1\leq l \leq n-1}p^l
<q^{n+1}\sum_{1\leq l \leq n} p^l
<(pq)^{n+1}.$$

This completes the proof.

\end{proof}

\section{Some simple applications}

The particular case $n=2$ of Theorem \ref{Honda} was already applied in several papers.

\medskip

Consider an extension of $C_{p^k}$ by an elementary abelian $p$-group. That is $G$ fits into an extension
$$
1\to(C_p)^l\to G\to C_{p^k}\to 1.
$$

It is known \cite{Y3,K} that $G$ is good, i.e., $K(s)^*(BG)$ is generated by Chern classes. However the explicit account of the ring structure was never done for $k>1$. 

The examples for the case $k=1$ was considered in \cite{BJ-JKT}, \cite{B1}. Namely let $\xi$ be a complex $m$-plane bundle over the total space of a cyclic 
covering $\pi :X\to X/C_p $ of prime index $p$. Let $c$ be the Chern class of $X\times_{C_p} \mathbb{C}\to X/C_p$, the complex line bundle associated to covering $\pi$ . In \cite{BP} we showed that modulo image of the transfer homomorphism the $i$-th Chern class $c_i$ of the transferred bundle $\xi$ can be written as a polynomial $\mathcal{A}_i$
in Chern classes $c_p,c_{2p}, \cdots c_{mp}$ and $c^{p-1}$. Using the polynomials $\mathcal{A}_i$ in \cite{BJ-JKT}, \cite{B1},
we showed for various examples of finite groups that $K(s)^*(BG)$ is the quotient of a polynomial ring by an ideal for which we listed explicit generators.

We recall that  Morava $K$-theory for a cyclic group is the truncated polynomials \cite{RAV}. In particular
$$K(s)^*(BC_{p^k})=\mathbb{F}_p[v_s,v_s^{-1}][u]/u^{p^{ks}}.$$
Also
$$K(s)^*(BU(m))=\mathbb{F}_p[v_s,v_s^{-1}][[c_1,\dots ,c_m]]$$
and because of the K\"{u}nneth isomorphisms
$$K(s)^*(BU(m) \times BC_{p^k}) = K(s)^*(BU(m)) \otimes K(s)^*(BC_{p^k}).$$

Theorem \ref{Honda} enables to write explicitly the relations derived by formal group law and splitting principle as relations in Chern classes of complex representations.

In particular, let $\theta$ be the line complex bundle over $BG$, associated to covering $\pi:BH\to BG$,\,\, $H=(C_p)^l$,
$\eta$
is the pullback by projection on the first factor $H\to C_p$ of the canonical bundle over $BC_p$ and
$\pi_!\eta$ is the transferred $\eta$. Then we have the bundle relation over $BG$
\begin{equation}
\label{fransferrelation}
\pi_!\eta\otimes \theta =\pi_!\eta.
\end{equation}

The relation \eqref{fransferrelation} holds because of Frobenius reciprocity of the transfer homomorphism of covering $\pi$ in complex $K$-theory:
$$
\pi_!\eta\otimes \theta=\pi_!(\eta \otimes \pi^*(\theta))=\pi_!(\eta \otimes 1)=\pi_!\eta.
$$

This implies the relations
$$
c_i(\pi_!\eta\otimes \theta)=c_i(\pi_!\eta)
$$
in $K(s)^*(BG)$. If we want to write everything in the explicit form, we have to apply the splitting principle to \eqref{fransferrelation} and write formally $\pi_!\eta$ as the sum of line bundles $\pi_!\eta=\eta_1+\cdots+\eta_{p^k}$. Thus we have
\begin{equation*}
\label{splitting}
\eta_1+\cdots+\eta_{p^k}=\eta_1\otimes \theta+\cdots+\eta_{p^k}\otimes \theta.
\end{equation*}

Using the elementary symmetric polynomials $\sigma_i$, $i=1,\cdots,p^k$ we can write
\begin{equation*}
\label{sigma}
c_i(\pi_!\eta)=\sigma_i(c_1(\eta_1),\cdots ,c_1(\eta_{p^k})).
\end{equation*}

In fact we have the following equations
\begin{equation}
\label{formal}
\sigma_i(c_1(\eta_1),\cdots ,c_1(\eta_{p^k}))=\sigma_i(F(c_1(\eta_1),c_1(\theta)),\cdots,F(c_1(\eta_{p^k}),c_1(\theta))).
\end{equation}

To rewrite \eqref{formal} explicitly, we apply Theorem \ref{Honda} for each term $F(c_1(\eta_j),c_1(\theta))$ and write it as a polynomial in $c_1(\eta_j)$ and  $u=c_1(\theta)$ as $\theta ^{p^k}=1$ implies $u^{p^{ks}}=0$. This is because $c_1(\theta^{p^k})=[p^k](c_1(\theta))$ and
$[p](x)=x^{p^{s}}$ for the Honda formal group law. 

Finally we turn to the Chern classes. In this way one can try to compute $K(s)^*(BG)$ as a quotient of a polynomial ring (as $G$ is finite) by relations ideal. For this we have to establish two facts: the classes we define generate, and the list of relations is complete. To check the latter is easier if the relations are given as explicit polynomials.

\bigskip

\section*{Acknowledgments} The author would like to thank the editors and the referee for several helpful comments. 

\bigskip

\bibliographystyle{amsplain}

\end{document}